\documentclass[11pt,a4paper,reqno]{amsart}

\usepackage[utf8]{inputenc}
\usepackage{amsmath, amssymb, amsthm, amsxtra}
\usepackage{mathtools}
\usepackage{mathrsfs}
\usepackage[all]{xy}
\usepackage{tikz-cd}
\usetikzlibrary{hobby}
\usepackage[shortlabels]{enumitem}
\usepackage{tensor}
\usepackage{xcolor}
\usepackage{caption}
\usepackage{microtype}
\usepackage{etoolbox}
\usepackage{fullpage}

\usepackage[
    backend=biber,
    style=numeric,        
    maxbibnames=99,       
    doi=true,      
    eprint=true,   
    isbn=false,    % Hides ISBN
    url=true       % Allows URL only if needed
]{biblatex}

\AtEveryBibitem{
  \iffieldundef{doi}{}{\clearfield{url}}
  \clearfield{issn}
}
\usepackage{hyperref}
\hypersetup{
    colorlinks=true,
    linkcolor=blue,
    citecolor=blue,
    urlcolor=blue
}

\newtheoremstyle{slantedplain}{}{}{\slshape}{}{\bfseries}{.}{.5em}{}

\theoremstyle{slantedplain}
\newtheorem{theorem}{Theorem}[section]
\newtheorem{proposition}[theorem]{Proposition}
\newtheorem{lemma}[theorem]{Lemma}
\newtheorem{corollary}[theorem]{Corollary}

\theoremstyle{definition}
\newtheorem{definition}[theorem]{Definition}
\newtheorem{remark}[theorem]{Remark}

\newtheorem{example}[theorem]{Example}

\newtheoremstyle{smallcaps}{}{0pt}{\normalfont}{}{\scshape}{.}{.5em}{}

\theoremstyle{smallcaps}

\AtEndEnvironment{proof}{
  \setcounter{case}{0}
  \setcounter{step}{0}
}

\setlist[enumerate,1]{label=\textnormal{(\roman*)}}
\setlist[description]{font=\normalfont\emph}

\DeclareMathOperator{\Aut}{Aut}

\DeclareMathOperator{\id}{id}

\DeclareMathOperator{\Spec}{Spec}

\DeclareMathOperator{\Graph}{Graph}
\DeclareMathOperator{\Char}{char} 

\newcommand{\ZZ}{\mathbb{Z}}
\newcommand{\QQ}{\mathbb{Q}}

\newcommand{\CC}{\mathbb{C}}
\newcommand{\PP}{\mathbb{P}}
\newcommand{\Hc}{\mathcal{H}}
\newcommand{\Hcbar}{\overline{\mathcal{H}}}
\newcommand{\Mcbar}{\overline{\mathcal{M}}}
\renewcommand{\Mc}{\mathcal{M}}

\makeatletter
\newcommand{\customlabel}[2]{%
   \protected@edef\@currentlabel{#2}\label{#1}%
}
\makeatother

\begin{document}

% --- Metadata ---
\title{Hyperelliptic Stable Curves}

\author{Max Schwegele}
\address{Institut für Algebra und Zahlentheorie, Universität Ulm, Helmholtzstrasse 18, D-89081 Ulm, Germany}
\email{max.schwegele@uni-ulm.de}

\date{\today} 

\begin{abstract}
    We provide an intrinsic characterization of hyperelliptic stable curves of genus $g \geq 2$, independent of admissible covers or auxiliary moduli data. A stable curve is hyperelliptic if it admits an involution yielding a rational tree quotient, subject to a characteristic-dependent condition. By analyzing the action of this involution on the nodes and decomposing the curve based on its connectivity, we obtain an explicit structural description of hyperellipticity and prove that the hyperelliptic involution is unique. Furthermore, we explain the connection to the very ampleness of the dualizing sheaf. This framework applies in arbitrary characteristic, explicitly capturing the divergent geometric and combinatorial behavior in characteristic $2$. We verify that this formulation precisely captures the geometric points of the moduli stack of hyperelliptic stable curves $\Hcbar_g$, defined as the scheme-theoretic closure of the smooth hyperelliptic locus $\Hc_g$ within the moduli stack of stable curves $\Mcbar_g$. Extending this definition to flat families yields an explicit modular description of $\Hcbar_g$ over $\Spec \ZZ[1/2]$.
\end{abstract}

\maketitle

% --- Main Content ---
\section{Introduction}

Let $g \geq 2$. We denote by $\mathcal{M}_g$ the moduli stack of smooth curves of genus $g$ and by $\overline{\mathcal{M}}_g$ its Deligne--Mumford compactification, which parametrizes stable curves. Within $\overline{\mathcal{M}}_g$, we consider the locus $\mathcal{H}_g$ of smooth hyperelliptic curves. We define the moduli stack of hyperelliptic stable curves, denoted $\overline{\mathcal{H}}_g$, as the scheme-theoretic closure of $\mathcal{H}_g$ in $\overline{\mathcal{M}}_g$. A stable curve $C$ over an algebraically closed field $k$ is called hyperelliptic if it corresponds to a geometric point of $\overline{\mathcal{H}}_g$.

Our main objective is to provide an intrinsic and practical criterion for when a stable curve $C$ is hyperelliptic. This is motivated by the fact that standard definitions in the literature can be cumbersome for concrete geometric applications. Many sources adopt a moduli-theoretic approach to define hyperelliptic stable curves (see \cite{cornalba1988divisor, harris1998moduli, ACGII}). This perspective depends on extrinsic data, making it impractical when one wishes to verify hyperellipticity using only the geometry of a given curve. An alternative approach utilizes the theory of admissible covers (see \cite{harris1982kodaira, harris1998moduli, abramovich2003twisted}). While more geometric, it shares a similar drawback: it focuses on the degeneration of the cover, where the source of the degenerated cover is generally not the stable limit itself. Instead, one must subsequently contract components to obtain the actual stable curve (see \cite[Theorem 3.160]{harris1998moduli}). Our intrinsic approach circumvents these issues, allowing one to verify hyperellipticity directly from the geometry of a concrete curve without relying on such auxiliary data.

To achieve this via a systematic decomposition, it is natural to work in the more general setting of pointed stable curves. More precisely, we say a stable pointed curve $(C, P)$ is \emph{hyperelliptic} if it admits an involution $\sigma \in \Aut_k(C, P)$ such that the quotient $T := C/\langle \sigma \rangle$ is a rational tree (i.e., a semistable curve of arithmetic genus $0$), and, if $\Char k \neq 2$, the involution $\sigma$ does not act as the identity on any irreducible component of $C$. Such an automorphism $\sigma$ is called a \emph{hyperelliptic involution}.

With this definition in place, our main results are the following:

\begin{theorem}[see Theorem~\ref{thm:uniqueness_hyperelliptic}]
    The hyperelliptic involution on a hyperelliptic stable pointed curve is uniquely determined.
\end{theorem}

Our inductive proof of this uniqueness naturally yields a canonical decomposition, provided that the curve is of semicompact type---a combinatorial condition on the global connectivity of its nodes. This procedure decomposes the curve into a collection of $2$-inseparable components, i.e., curves with $3$-edge-connected dual graphs. Utilizing this, we obtain a complete structural characterization:

\begin{theorem}[see Theorem~\ref{thm:structure_hyperelliptic}]
    A stable pointed curve $(C, P)$ is hyperelliptic if and only if it is of semicompact type and all its $2$-inseparable components are hyperelliptic. For these components to be hyperelliptic, it is necessary that they are either irreducible or binary curves (the union of two smooth rational components intersecting at multiple nodes), and their hyperellipticity can be explicitly characterized via their normalization, or alternatively, via the dualizing sheaf of an associated pinched curve.
\end{theorem}

Finally, we establish that our intrinsic geometric definition perfectly aligns with the moduli-theoretic perspective:

\begin{theorem}[see Theorem~\ref{thm:geometric_points}]
    A stable curve $C$ over an algebraically closed field $k$ represents a geometric point of the stack closure $\overline{\mathcal{H}}_g$ if and only if $C$ is a hyperelliptic stable curve in the sense of the above definition.
\end{theorem}

Generalizing the notion of hyperelliptic curves to families, this approach yields a concrete modular description of the moduli stack $\overline{\mathcal{H}}_g$ over $\Spec \ZZ[1/2]$ as a substack of $\overline{\mathcal{M}}_g$. Normally, this stack is described via the theory of admissible covers.

These results are not entirely new and are likely mostly known to experts. Variations of our intrinsic definition of hyperelliptic stable curves have appeared in the literature. Over the complex numbers $\CC$, an equivalent definition is used by Arbarello, Cornalba, and Griffiths \cite[Chapter X, Section 3]{ACGII}. It is well understood in the literature that this definition extends without complications to arbitrary algebraically closed fields with $\Char k \neq 2$ (an explicit formulation can be found, for example, in \cite[Definition 4.1]{kawaguchi2015rank}). The general case, including characteristic $2$, first appears in the work of Maugeais \cite[Definition 3.12]{maugeais2003relevement} (we refer to Remark~\ref{rem:literature_definition} for a detailed comparison of these definitions across the literature). Regarding the uniqueness of the involution, a proof for unpointed curves over $\CC$ was given in \cite[Chapter X, Lemma 3.5]{ACGII}, proceeding by induction on the number of components of the quotient tree $T$. Our approach, by contrast, inducts directly on the separating nodes and pairs of $C$. The terminology for this decomposition strategy---such as \emph{semicompact} curves and \emph{$2$-inseparable} components---is borrowed from Ran \cite{ran2014canonical}, though our intrinsic geometric methods remain entirely independent. This direct focus offers the distinct advantage of allowing one to systematically deduce the potential hyperellipticity of a combinatorial type without needing \emph{a priori} knowledge of the quotient map to $T$.

Unlike most sources in the literature, our framework explicitly incorporates fields of characteristic $2$, allowing us to thoroughly investigate the divergent behavior of hyperelliptic stable curves in this setting. By dropping the second condition of our definition---namely, the requirement that the involution does not act as the identity on any irreducible component---we admit new geometric configurations. Proposition~\ref{prop:hyperelliptic_genus2} precisely explains which combinatorial types are affected by this relaxation, giving rise to cases that are exclusively hyperelliptic in characteristic $2$. Note that since $\overline{\mathcal{M}}_2 = \overline{\mathcal{H}}_2$, every stable curve of genus $2$ is automatically hyperelliptic. Thus, genus $g=3$ represents the first case where these characteristic-dependent phenomena can be observed. The phenomenon described by Proposition~\ref{prop:hyperelliptic_genus2} is illustrated for a genus $3$ curve in Example~\ref{ex:char2_exclusive}. Conversely, as a consequence of wild ramification, certain combinatorial types are forced to be strictly non-hyperelliptic in characteristic $2$; we demonstrate this in Example~\ref{ex:char2_forbidden}. An exhaustive analysis of all $g=3$ cases can be found in our joint work on the semistable reduction of plane quartics \cite[Proposition 1.14]{quartic_paper}.

\medskip\noindent
\textbf{Structure of the paper.} In Section~\ref{sec:stable_curves}, we fix the notation and review the combinatorial properties of stable curves and their dual graphs. Section~\ref{sec:involution} is devoted to the definition of the hyperelliptic involution and provides a complete geometric classification of the nodes based on the action of the involution. Section~\ref{sec:uniqueness} proves the uniqueness of the hyperelliptic involution using an inductive decomposition approach and derives the structural characterization. Finally, Section~\ref{sec:moduli} extends these concepts to flat families over arbitrary base schemes, examines the geometric points of the moduli stack $\overline{\mathcal{H}}_g$, and concludes with the combinatorial stratification of the boundary.
\section{Stable Curves and Dual Graphs} \label{sec:stable_curves}

In this section, we fix the notation and review the combinatorial properties of stable curves necessary for our subsequent decomposition results. We fix an algebraically closed base field $k$, and all curves are assumed to be defined over $k$.

\subsection{Stable Pointed Curves}

We begin by establishing our conventions for semistable curves and their stability conditions in the pointed setting. Throughout this paper, by a \emph{semistable curve}, we mean a connected, projective, reduced scheme of dimension one over $k$ whose only singularities are ordinary double points (nodes). We denote the arithmetic genus of a semistable curve $C$ by $g(C) := p_a(C)$.

\begin{definition} \label{def:stable_pointed_curve}
    An $n$-pointed semistable curve is a pair $(C, P)$ consisting of a semistable curve $C$ and an $n$-tuple $P = (p_1, \dots, p_n)$ of distinct smooth points on $C$. We call $(C, P)$ a \emph{stable $n$-pointed curve} if the automorphism group $\Aut_k(C, P)$---the group of automorphisms of $C$ that fix each point $p_i \in P$---is finite. When $n=0$, we omit the tuple $P$ and simply refer to the stable curve $C$.
\end{definition}

This finiteness condition is equivalent to the ampleness of the log-canonical sheaf $\omega_{C/k}(P)$, where $P$ denotes the divisor $p_1 + \dots + p_n$. More explicitly, this requires that $2g(Z) - 2 + s_Z > 0$ for every irreducible component $Z$ of $C$, where $s_Z$ denotes the number of special points on $Z$ (the marked points and the nodes connecting $Z$ to the rest of the curve). Observe that this condition is non-trivial only for components of arithmetic genus $g(Z) = 0$ or $g(Z) = 1$.

For local analysis, it is often necessary to decompose a pointed semistable curve $(C, P)$ at a set of nodes $S$. Let $\nu\colon \tilde{C} \to C$ denote the partial normalization of $C$ at $S$, which induces a decomposition of $\tilde{C}$ into connected components $\tilde{C} = C_1 \sqcup \dots \sqcup C_m$. We naturally identify the smooth points in $P$ with their unique preimages in $\tilde{C}$. By equipping each component $C_i$ with the marked points $P_i := (\nu^{-1}(S) \cap C_i) \cup (P \cap C_i)$, we obtain the \emph{pointed decomposition} $\{(C_1, P_1), \dots, (C_m, P_m)\}$ of $(C,P)$ induced by $S$.

\begin{lemma} \label{lem:stable_decomposition}
    Let $(C, P)$ be a pointed semistable curve, and let $S$ be a set of nodes of $C$. Then $(C, P)$ is stable if and only if every component $(C_i, P_i)$ in the induced pointed decomposition is stable.
\end{lemma}

\begin{proof}
    Recall that stability is equivalent to the ampleness of the log-canonical sheaf $\omega_{C/k}(P)$. Since the partial normalization $\nu\colon \tilde{C} \to C$ is a finite surjective morphism, $\omega_{C/k}(P)$ is ample on $C$ if and only if its pullback $\nu^*\omega_{C/k}(P)$ is ample on $\tilde{C}$. The conclusion then follows from the canonical isomorphisms $\nu^*\omega_{C/k}(P)|_{C_i} \cong \omega_{C_i/k}(P_i)$ and the fact that ampleness on a disjoint union is equivalent to ampleness on each individual connected component.
\end{proof}

\subsection{Graph of a Decomposition}

Given a decomposition of $(C,P)$ induced by a set of nodes $S$, we associate to it a graph denoted $\Gamma = \Graph^S(C,P)$. In the case where $S$ constitutes the set of all nodes of $C$, the graph $\Gamma$ coincides with the standard dual graph. The vertices of $\Gamma$ correspond to the components $V = \{C_1, \dots, C_m\}$ of the decomposition. The edges of $\Gamma$ correspond to the nodes in $S$, with an edge connecting $C_i$ and $C_j$ if the corresponding node joins these two components. Additionally, the marked points in $P \cap C_i$ are represented as half-edges (or legs) attached to the vertex $C_i$. To retain the geometric data, each vertex $v \in V$ is equipped with a weight $g_v := g(C_i)$, recording the arithmetic genus of the corresponding component. With this data, the arithmetic genus of the original curve $C$ can be recovered as $g(C) = \sum_{v \in V} g_v + h^1(\Gamma)$, where $h^1(\Gamma)$ denotes the first Betti number of the graph $\Gamma$.

\subsection{Connectivity and Semicompact Type}
We classify the nodes of a semistable curve based on their impact on its overall connectivity.

\begin{definition} \label{def:comb_term}
Let $C$ be a semistable curve over $k$.
\begin{enumerate}
    \item A node $p$ is \emph{separating} if the partial normalization of $C$ at $p$ disconnects the curve. (Equivalently, the corresponding edge in the dual graph is a bridge.)
    \item A pair of distinct nodes $\{p, q\}$ is a \emph{separating pair} if the partial normalization of $C$ at $\{p, q\}$ disconnects the curve, but the partial normalization at either $p$ or $q$ individually does not. (Equivalently, the corresponding edges form a minimal 2-edge-cut in the dual graph.)
    \item A node $p$ is \emph{strongly non-separating} if it is neither separating nor part of a separating pair.
    \item The curve $C$ is \emph{inseparable}\footnote{This terminology is inspired by \cite{ran2014canonical}. This geometric notion should not be confused with the algebraic concept of an inseparable field extension or morphism.} if it has no separating nodes, and is \emph{separable} otherwise. It is \emph{2-inseparable} if it is inseparable and contains no separating pairs. (Combinatorially, these conditions correspond to the dual graph being 2-edge-connected and 3-edge-connected, respectively.)
    \item The curve $C$ is of \emph{semicompact type} if every node belongs to at most one separating pair.
\end{enumerate}
\end{definition}
\section{Hyperelliptic Stable Curves} \label{sec:involution}

In this section, we formally define hyperelliptic stable pointed curves and provide a complete geometric classification of their nodes based on the action of the hyperelliptic involution.

\subsection{The Hyperelliptic Involution}

In the smooth setting, hyperellipticity is characterized by the existence of an involution whose quotient is isomorphic to $\PP^1$. For stable curves, the target degenerates into a tree of rational curves.

\begin{definition} \label{def:hyperelliptic}
    A stable pointed curve $(C, P)$ is \emph{hyperelliptic} if there exists an involution $\sigma \in \Aut_k(C, P)$ such that:
    \begin{enumerate}
        \item the quotient $T := C/\langle \sigma \rangle$ is a rational tree (i.e., a semistable curve of arithmetic genus $0$); and
        \item \label{def:hyperelliptic_ii} if $\Char k \neq 2$, the involution $\sigma$ does not act as the identity on any irreducible component of $C$.
    \end{enumerate}
    Such an automorphism $\sigma$ is called a \emph{hyperelliptic involution}.
\end{definition}

Note that by \cite[Appendice, Corollaire de la Proposition 5]{raynaud2013p}, the quotient $T = C/\langle \sigma \rangle$ of a semistable curve by an involution is always semistable.

\begin{remark} \label{rem:literature_definition}
    The definition of hyperellipticity for unpointed stable curves appears in various forms across the literature. Over $\CC$, this definition can be found in \cite[Chapter X, Section 3]{ACGII}, where Condition~\ref{def:hyperelliptic_ii} is replaced by the equivalent requirement that $\sigma$ has only isolated fixed points. For arbitrary algebraically closed fields of characteristic different from $2$, it is articulated in \cite[Definition 4.1]{kawaguchi2015rank}. To the best of our knowledge, the general case over fields of arbitrary characteristic first appears in \cite[Definition 3.12]{maugeais2003relevement}. We note that the ``kummerienne'' condition imposed there for $2$-coverings in characteristic $\neq 2$ is precisely equivalent to our Condition~\ref{def:hyperelliptic_ii}. Another reference is \cite[Remark 1.2]{yamaki2004cornalba}, although Condition~\ref{def:hyperelliptic_ii} appears to have been inadvertently omitted in that text.
\end{remark}

We refer to Section~\ref{sec:moduli} for an explanation of why this formulation constitutes the correct intrinsic definition from a moduli-theoretic perspective.

\subsection{Classification of Nodes}

The nodes of a hyperelliptic stable curve admit a natural classification into distinct types (cf.\ \cite[p. 102]{ACGII}). This classification is well established in the study of the boundary divisors of $\Hcbar_g$ (cf.\ \cite{cornalba2006picard}). Although we do not rely on these moduli-theoretic connections, we adopt the standard notation for these node types to maintain consistency with the literature.

\begin{proposition} \label{prop:node_classification}
    Let $(C, P)$ be a hyperelliptic stable pointed curve with hyperelliptic involution $\sigma$ and quotient map $\pi\colon C \to T$, and let $p$ be a node of $C$. Then $p$ falls into exactly one of the three combinatorial types from Definition~\ref{def:comb_term}, which are characterized by the action of $\sigma$ as follows:
    \begin{enumerate}
        \item[$(\Xi)$] \customlabel{item:properties_nodes_sep}{$\Xi$} \textbf{Separating Node.} The node $p$ is separating if and only if it is a fixed point of $\sigma$ at which $\sigma$ preserves the local branches. In this case, its image $\pi(p)$ is a node of $T$, and $\sigma$ respects the decomposition of $C$ induced by $p$.
        
        \item[$(\Delta)$] \customlabel{item:properties_nodes_pair}{$\Delta$} \textbf{Part of a Separating Pair.} The node $p$ belongs to a separating pair if and only if it is not a fixed point of $\sigma$. In this case, its image under the involution, $p' := \sigma(p)$, is the unique node that forms a separating pair $\{p, p'\}$ with $p$. Both points map to the same node $\pi(p) = \pi(p')$ in $T$, and $\sigma$ respects the decomposition of $C$ induced by the pair $\{p, p'\}$.
        
        \item[$(\Xi_{\mathrm{irr}})$] \customlabel{item:properties_nodes_non}{$\Xi_{\mathrm{irr}}$} \textbf{Strongly Non-Separating Node.} The node $p$ is strongly non-separating if and only if it is a fixed point of $\sigma$ at which $\sigma$ interchanges the local branches. In this case, its image $\pi(p)$ is a smooth point of $T$, and the irreducible components of $C$ forming the node $p$ are either identical or both rational.
    \end{enumerate}
    In particular, the curve $C$ is of semicompact type. Moreover, the rational tree $T = C/\langle \sigma \rangle$ is irreducible (i.e., $T \simeq \PP^1_k$) if and only if $C$ is $2$-inseparable.\footnote{This fact appears to be folklore, and we are unaware of an explicit reference; cf.\ \cite[Remark 2.14]{ran2014canonical}.}
\end{proposition}

\begin{remark}
   The quotient morphism $\pi \colon C \to T$ is finite and, provided $\sigma$ is not globally the identity, forms a $2$-covering in the sense of Maugeais \cite[Definition 3.11]{maugeais2003relevement}. It has degree $2$ over every irreducible component where $\sigma$ does not act as the identity, and degree $1$ otherwise (which is only possible if $\Char k = 2$).
   
   In the tame setting ($\Char k \neq 2$), $\pi$ is an admissible cover (introduced by Harris and Mumford \cite{harris1982kodaira}; see \cite{abramovich2003twisted} for a modern treatment) if and only if $C$ contains no nodes of type \ref{item:properties_nodes_non}. Geometrically, such nodes arise when two Weierstrass points collide during the degeneration of a smooth hyperelliptic curve. To obtain an admissible cover in this scenario, one modifies the curve by inserting rational bridges at these nodes in $C$ and corresponding rational components at their images in $T$ (cf.\ \cite[Section 4]{harris1982kodaira} and \cite[Section 3.G]{harris1998moduli}). 
   
   In the wild setting ($\Char k = 2$), the classical theory of admissible covers does not apply. Instead, \cite[Définition 3.15]{maugeais2003relevement} provides the analogous notion of $2$-admissible covers. Upgrading $\pi$ to a $2$-admissible cover again requires inserting rational bridges at nodes of type \ref{item:properties_nodes_non}. These insertions are termed ``modifications'' in \cite[Definition 3.17 and Proposition 5.1]{maugeais2003relevement}. Additionally, to achieve a uniform degree of $2$ over components where $\pi$ has degree $1$, one composes the map with the relative Frobenius morphism to obtain a finite map $\pi' \colon C \to T'$ of degree $2$, which then is part of the data of the $2$-admissible cover (see \cite[Proposition 5.1]{maugeais2003relevement} and \cite[Appendix A.2]{yamaki2004cornalba}).
\end{remark}

To prove Proposition~\ref{prop:node_classification}, we first establish two structural lemmas regarding decompositions. The first of these clarifies the terminology ``respects the decomposition'' used in the preceding classification.

\begin{lemma} \label{lem:orbit_decomposition_preserved}
    Let $(C, P)$ be a hyperelliptic stable pointed curve with hyperelliptic involution $\sigma$. Let $S$ be a $\sigma$-invariant set of nodes, and let $\nu\colon \tilde{C} \to C$ denote the partial normalization of $C$ at $S$, with induced pointed decomposition $\{(C_1, P_1), \dots, (C_m, P_m)\}$. For each $i$, set $S_i := \nu^{-1}(S) \cap C_i$.
    \begin{enumerate}
        \item \label{item:orbit_decomposition_preserved_i} The involution $\sigma$ lifts to a unique involution on $\tilde{C}$, still denoted by $\sigma$, which permutes the components $C_1, \dots, C_m$. 
        \item \label{item:orbit_decomposition_preserved_ii} If $\sigma$ interchanges $C_i$ and $C_j$ for $i \neq j$, then $C_i$ is a rational tree and $|P_i| = |S_i| \geq 3$. Consequently, any component $C_i$ with $|S_i| \leq 2$ must be fixed by $\sigma$.
        \item \label{item:orbit_decomposition_preserved_iii} If $|S| \leq 2$, the involution $\sigma$ \emph{respects the decomposition}; that is, it fixes every component of the decomposition.
    \end{enumerate}
\end{lemma}

\begin{proof}
    \ref{item:orbit_decomposition_preserved_i} By the universal property of normalization, the involution $\sigma$ lifts to a unique involution on $\tilde{C}$ that permutes its connected components.
    
    \ref{item:orbit_decomposition_preserved_ii} Suppose $\sigma(C_i) = C_j$ for $i \neq j$. Since $\sigma$ fixes $P$ pointwise, we must have $P \cap C_i = \emptyset$, and hence $P_i = S_i$. Furthermore, the quotient $\tilde{C}/\langle\sigma\rangle$ maps via normalization to the rational tree $C/\langle\sigma\rangle$. Because $\sigma$ swaps the disjoint components $C_i$ and $C_j$, the quotient map restricts to an isomorphism on $C_i$, meaning each component $C_i$ is a rational tree. Since the components of a decomposed stable curve remain stable, the stability of $(C_i, P_i)$ requires that $|S_i| = |P_i| \geq 3$.
    
    \ref{item:orbit_decomposition_preserved_iii} If $|S| \leq 2$, then $\sigma$ interchanging $C_i$ and $C_j$ would require $|S_i| + |S_j| \geq 6$. This contradicts the fact that each node in $S$ contributes exactly two preimages, meaning $\sum_l |S_l| = 2|S| \leq 4$. Thus, $\sigma$ must fix every component.
\end{proof}

\begin{lemma} \label{lem:chains_of_cuts}
    Let $S = \{p_1, \dots, p_n\}$ be distinct nodes of a semistable curve $C$, inducing a decomposition $\{(C_1, S_1), \dots, (C_m, S_m)\}$. In both cases below, the two preimages of each $p \in S$ lie on distinct components. This allows us to naturally identify each $C_i$ with a subcurve of $C$, and each $S_i$ with a subset of $S$.
    \begin{enumerate}
        \item \label{item:chains_of_cuts_i} If every $p \in S$ is a separating node, then $m = n+1$ and the components form a tree (i.e., any two components are connected by a unique chain of nodes in $S$).
        \item \label{item:chains_of_cuts_ii} If $n \geq 2$ and $\{p_i, p_{i+1}\}$ is a separating pair for all $i < n$, then every pair of distinct nodes $\{p_i, p_j\}$ in $S$ is separating. In this case, $m = n$ and the components form a cycle (understood as two parallel edges if $n=2$). Up to relabeling, $p_i$ connects $C_i$ to $C_{i+1}$, and $S_i = \{p_{i-1}, p_i\}$ (with indices taken modulo $n$).
    \end{enumerate}
\end{lemma}

\begin{proof}
    We translate the problem to the associated graph $\Gamma := \Graph^S(C,\emptyset)$, whose vertices and edges correspond to the components $C_i$ and the nodes in $S$, respectively. Since $C$ is connected, $\Gamma$ is connected. Observe that the two preimages of a node in $S$ lie on distinct components if and only if the corresponding edge in $\Gamma$ is not a loop.
    
    \ref{item:chains_of_cuts_i} The assumption means every edge in $\Gamma$ is a bridge. A connected graph in which every edge is a bridge is a tree. Therefore, $\Gamma$ has no loops, possesses $n$ edges, and has $m = n+1$ vertices, matching the claimed tree structure.
    
    \ref{item:chains_of_cuts_ii} The assumption translates to the pairs of edges $\{p_i, p_{i+1}\}$ forming minimal $2$-edge-cuts in $\Gamma$. Because every edge $p_i$ belongs to such a cut, $\Gamma$ contains no bridges or loops and is therefore $2$-edge-connected. Recall that the symmetric difference of two edge-cuts is again an edge-cut. Since $\Gamma$ has no bridges, any such $2$-edge-cut is automatically minimal. Applying this iteratively shows that \emph{any} pair of distinct edges $\{p_i, p_j\}$ forms a minimal $2$-edge-cut. This strict connectivity condition forces every vertex in $\Gamma$ to have degree exactly $2$ (we refer to \cite[Lemma 2.21]{MaxMaster} for the detailed contradiction argument proving this degree bound). A connected graph where every vertex has degree $2$ is an $n$-cycle, yielding $m=n$ components arranged cyclically with no loops.
\end{proof}

\begin{proof}[Proof of Proposition \ref{prop:node_classification}]
    Let $p$ be a node of $C$. 

    (\ref{item:properties_nodes_sep}) Assume $p$ is a separating node. If $p \neq \sigma(p)$, the set $S=\{p, \sigma(p)\}$ induces a chain of three components (Lemma~\ref{lem:chains_of_cuts}~\ref{item:chains_of_cuts_i}). The involution $\sigma$ must swap the two outer components, which contradicts Lemma~\ref{lem:orbit_decomposition_preserved}~\ref{item:orbit_decomposition_preserved_iii} stating that $\sigma$ fixes all components when $|S|=2$. Thus, $p = \sigma(p)$. Since $\sigma$ fixes $p$ and the adjacent components, it preserves the branches. Conversely, if $\sigma$ fixes $p$ and its branches, $\pi(p)$ is a node in the quotient tree $T$. Since any node in a tree is separating and the fiber over $\pi(p)$ is exactly $\{p\}$, the partial normalization of $C$ at $p$ disconnects the curve.

    (\ref{item:properties_nodes_pair}) Assume $p \neq \sigma(p)$. By (\ref{item:properties_nodes_sep}), neither node is a separating node. The fiber over $\pi(p)$ is exactly $\{p, \sigma(p)\}$, so $\pi(p)$ is a node of $T$. The partial normalization of $T$ at $\pi(p)$ disconnects the tree, meaning $\{p, \sigma(p)\}$ forms a separating pair in $C$. 
    Conversely, let $\{p, q\}$ be a separating pair. Consider $S = \{p, q, \sigma(p), \sigma(q)\}$. Since $\{p, q\}$ is a separating pair, so is its image $\{\sigma(p), \sigma(q)\}$. Furthermore, as established above, if $p \neq \sigma(p)$, then $\{p, \sigma(p)\}$ is a separating pair (and likewise for $q$ if $q \neq \sigma(q)$). Thus, the assumptions of Lemma~\ref{lem:chains_of_cuts}~\ref{item:chains_of_cuts_ii} are satisfied. If $|S| > 2$, the lemma dictates that $S$ induces a cycle of $m \ge 3$ components. Because each component $C_i$ in this cycle contains exactly two nodes from $S$ (i.e., $|S_i| = 2$), Lemma~\ref{lem:orbit_decomposition_preserved}~\ref{item:orbit_decomposition_preserved_ii} dictates that $\sigma$ must fix these components. This rigidly forces $\sigma$ to fix the connecting nodes individually, contradicting $|S|>2$. Thus, $S=\{p,q\}$. If $\sigma(p)=p$, then $\sigma$ must interchange its branches (since $p$ is non-separating by definition). This would swap the two components connected by $S$, contradicting Lemma~\ref{lem:orbit_decomposition_preserved}~\ref{item:orbit_decomposition_preserved_iii}. Therefore, we must have $\sigma(p)=q$, and by the same lemma, $\sigma$ respects the induced decomposition.

    (\ref{item:properties_nodes_non}) By the previous two items, $p$ is strongly non-separating if and only if it is fixed by $\sigma$ but does not preserve its branches (i.e., $\sigma$ interchanges them). Consequently, the quotient map locally folds these branches together, making $\pi(p)$ a smooth point of $T$. If $p$ lies on two distinct irreducible components $Z_1$ and $Z_2$, interchanging the branches means $\sigma(Z_1) = Z_2$. The quotient map restricts to an isomorphism from $Z_1$ to an irreducible subcurve of the rational tree $T$, forcing both $Z_1$ and $Z_2$ to be rational.

    Finally, we deduce the global properties of $C$. Because any node in a separating pair is swapped with its unique partner $\sigma(p)$, no node can belong to more than one separating pair, meaning $C$ is always of semicompact type. Furthermore, $C$ is $2$-inseparable if and only if it has no nodes of type (\ref{item:properties_nodes_sep}) or (\ref{item:properties_nodes_pair}). By our classification, this occurs if and only if $T$ has no nodes, which is equivalent to the rational tree $T$ being irreducible ($T \simeq \PP^1_k$).
\end{proof}

\section{Decomposition and Uniqueness} \label{sec:uniqueness}

The primary objective of this section is to prove the uniqueness of the hyperelliptic involution on a stable curve. We proceed by induction on the number of separating nodes and separating pairs, effectively reducing the problem to the $2$-inseparable base case. A proof of uniqueness for unpointed curves over $\CC$ appears in \cite[Chapter X, Lemma 3.5]{ACGII}, proceeding by induction on the number of irreducible components of the quotient tree $T$. Our reformulation, which instead inducts directly on the separating nodes and pairs of $C$, offers a distinct advantage: it allows one to systematically deduce whether a given combinatorial type is potentially hyperelliptic without requiring \emph{a priori} knowledge of the quotient map to $T$.

\begin{lemma}[Decomposition at a Separating Node] \label{lem:decomp_sep_node}
    Let $p$ be a separating node of a stable pointed curve $(C, P)$, and let $\{(C_1, P_1), (C_2, P_2)\}$ be the pointed decomposition of $(C, P)$ induced by $p$. Then $(C, P)$ is hyperelliptic if and only if both $(C_1, P_1)$ and $(C_2, P_2)$ are hyperelliptic.

    Furthermore, hyperelliptic involutions on $(C, P)$ correspond bijectively to pairs of hyperelliptic involutions on $(C_1, P_1)$ and $(C_2, P_2)$ via restriction and gluing.
\end{lemma}

\begin{proof}
    For $\Char k \neq 2$, the condition that an involution does not act as the identity on any irreducible component is preserved under restriction and gluing. Thus, we need only verify that the respective quotients are rational trees.

    If $\sigma$ is a hyperelliptic involution on $(C, P)$, it fixes the separating node $p$ by Proposition~\ref{prop:node_classification} and restricts to involutions $\sigma_i$ on $(C_i, P_i)$. Since the quotients $C_i/\langle\sigma_i\rangle$ are subcurves of the rational tree $C/\langle\sigma\rangle$, they are inherently rational trees, making the $\sigma_i$ hyperelliptic.

    Conversely, given hyperelliptic involutions $\sigma_i$ on $(C_i, P_i)$, they both fix the marked point $p \in P_i$ by definition, and therefore glue uniquely to an involution $\sigma$ on $(C, P)$. Its quotient $C/\langle\sigma\rangle$ joins the rational trees $C_i/\langle\sigma_i\rangle$ at a single node, forming a rational tree, whence $\sigma$ is hyperelliptic.
\end{proof}

To handle separating pairs, we formally construct the contracted components of the decomposition.

\begin{definition} \label{def:contracted_components}
    Let $S=\{p, q\}$ be a separating pair of a stable pointed curve $(C, P)$, inducing the pointed decomposition $\{(C_1, P_1), (C_2, P_2)\}$, and let $S_i \subset C_i$ denote the preimages of $S$. The \emph{contracted component} $(C_i', P_i')$ is the stable pointed curve obtained from $C_i$ by identifying the two points in $S_i$ to form a new node $r_i$, and defining the marked points as $P_i' := P_i \setminus S_i$.
\end{definition}

In other words, $(C_i, P_i)$ is precisely the pointed decomposition of $(C_i', P_i')$ induced by the non-separating node $r_i$. Consequently, the stability of the contracted component $(C_i', P_i')$ follows immediately from Lemma~\ref{lem:stable_decomposition}. Furthermore, if the original curve $C$ is of semicompact type, then $r_i$ is a strongly non-separating node of $C_i'$. Indeed, if $r_i$ were part of a separating pair in $C_i'$, the corresponding node in $C$ would form a separating pair with both $p$ and $q$, which is impossible since $C$ is of semicompact type.

\begin{lemma}[Decomposition at a Separating Pair] \label{lem:decomp_sep_pair}
    Let $S = \{p, q\}$ be a separating pair of a stable pointed curve $(C, P)$ of semicompact type, and let $\{(C_1, P_1), (C_2, P_2)\}$ be the pointed decomposition of $(C, P)$ induced by $S$. Then $(C, P)$ is hyperelliptic if and only if both contracted components $(C_1', P_1')$ and $(C_2', P_2')$ are hyperelliptic.

    Furthermore, hyperelliptic involutions on $(C, P)$ correspond bijectively to pairs of hyperelliptic involutions on the contracted components $(C_1', P_1')$ and $(C_2', P_2')$.
\end{lemma}

\begin{proof}
    For $\Char k \neq 2$, the condition that an involution does not act as the identity on any irreducible component is preserved under contraction and lifting. Thus, we need only verify that the respective quotients are rational trees.
    
    If $(C,P)$ is hyperelliptic with involution $\sigma$, Proposition~\ref{prop:node_classification} implies that $\sigma$ swaps $p$ and $q$ and respects the decomposition, fixing the components $C_i$. The restriction $\sigma_i$ swaps the points in $S_i$ and therefore descends to an involution $\sigma_i'$ on $C_i'$ fixing the new node $r_i$. The quotient $C_i'/\langle\sigma_i'\rangle \simeq C_i/\langle\sigma_i\rangle$ is a subcurve of the rational tree $C/\langle\sigma\rangle$, making it a rational tree.
    
    Conversely, hyperelliptic involutions $\sigma_i'$ on $C_i'$ must fix the strongly non-separating node $r_i$ and interchange its branches. Lifting $\sigma_i'$ to the partial normalization $C_i$ yields an involution $\sigma_i$ swapping the preimages $S_i$. Gluing these components along $S_i$ produces a global involution $\sigma$ on $(C, P)$ that swaps $p$ and $q$. The quotient $C/\langle\sigma\rangle$ is obtained by gluing two rational trees at a single smooth point from each, resulting in a rational tree, whence $\sigma$ is hyperelliptic.
\end{proof}

We now address the $2$-inseparable base case, demonstrating that such a curve admits at most one hyperelliptic involution.

\begin{lemma} \label{lem:unique_2_inse}
    Let $(C, P)$ be a $2$-inseparable hyperelliptic stable pointed curve. Then $C$ must be one of the following two types:
    \begin{enumerate}
        \item $C$ is irreducible; or
        \item $C$ is a \emph{binary curve}: the union of two smooth rational components intersecting at $g(C)+1 \ge 3$ nodes.
    \end{enumerate}
    In either case, the hyperelliptic involution on $(C, P)$ is unique.
\end{lemma}

\begin{proof}
    Let $\sigma, \tau$ be hyperelliptic involutions on $(C, P)$. By $2$-inseparability and Proposition~\ref{prop:node_classification}, every node of $C$ is a fixed point where both involutions interchange the local branches. Consequently, their lifts $\tilde{\sigma}, \tilde{\tau}$ to the normalization $\tilde{C}$ fix the preimages of $P$, interchange the preimages of each node, and yield rational quotients $\tilde{C}/\langle\tilde{\sigma}\rangle \simeq \tilde{C}/\langle\tilde{\tau}\rangle \simeq \PP^1_k$.

    If $C$ is irreducible, the uniqueness of the lifted involution on the smooth curve $\tilde{C}$ for $g(\tilde{C}) \geq 2$ follows from classical theory. If $g(\tilde{C}) = 1$, stability requires $\tilde{C}$ to have at least one marked point or the preimages of a node. If there is a marked point, both lifts fix it; taking this point as the origin of the group law, the unique involution yielding a genus $0$ quotient is the inversion morphism $[-1]$. If there are only nodes, both lifts must interchange the two preimages of a node. While a translation by a $2$-torsion point could also interchange these points, it would yield a genus $1$ quotient; the requirement of a $\PP^1_k$ quotient ensures the involution is the composition of the inversion morphism with a translation, which is uniquely determined by the interchanged pair. Finally, if $g(\tilde{C}) = 0$, stability requires $\tilde{C} \simeq \PP^1_k$ to have at least three special points. Since an automorphism of $\PP^1_k$ is uniquely determined by its action on three points, and both lifts act identically by fixing marked points and interchanging node preimages, they must coincide. Thus, $\tilde{\sigma} = \tilde{\tau}$.

    Now assume $C$ is reducible. Because the irreducible quotient $\PP^1_k$ forces a transitive action, and involutions generate orbits of size at most two, the reducible curve $C$ must consist of exactly two interchanged components. This precludes the existence of marked points (hence $n=0$) and forces $C_1 \simeq C_2 \simeq \PP^1_k$. For stability, they must intersect at $g(C)+1 \geq 3$ nodes. The restrictions $\tilde{\sigma}|_{C_1}$ and $\tilde{\tau}|_{C_1}$ are isomorphisms $C_1 \to C_2$ agreeing on these $\geq 3$ node preimages, hence $\tilde{\sigma} = \tilde{\tau}$. 

    In all cases, $\tilde{\sigma} = \tilde{\tau}$, which implies $\sigma = \tau$.
\end{proof}

\begin{theorem} \label{thm:uniqueness_hyperelliptic}
    The hyperelliptic involution on a hyperelliptic stable pointed curve is uniquely determined.
\end{theorem}

\begin{proof}
    This follows by induction on the number of separating nodes and separating pairs, using Lemmas~\ref{lem:decomp_sep_node} and \ref{lem:decomp_sep_pair} to reduce the problem to the $2$-inseparable base case of Lemma~\ref{lem:unique_2_inse}.
\end{proof}

The preceding proofs demonstrate that any stable pointed curve of semicompact type canonically decomposes into $2$-inseparable stable pointed curves, which we call its \emph{$2$-inseparable components}. This decomposition, together with our uniqueness proof, directly yields the following structural characterization of stable pointed hyperelliptic curves entirely in terms of smooth pointed curves.

\begin{theorem} \label{thm:structure_hyperelliptic}
    A stable pointed curve $(C, P)$ is hyperelliptic if and only if it is of semicompact type and all its $2$-inseparable components are hyperelliptic. 
    
    A $2$-inseparable stable $n$-pointed curve $(C, P)$ is hyperelliptic if and only if exactly one of the following holds:
    \begin{enumerate}
        \item $C$ is irreducible. Let $\nu\colon \tilde{C} \to C$ be its normalization, let $\tilde{P} = \nu^{-1}(P)$ be the preimages of the marked points, and let $\tilde{g} = g(\tilde{C})$. Depending on $\tilde{g}$, exactly one of the following holds:
        \begin{itemize}
            \item If $\tilde{g} \geq 2$: The smooth curve $\tilde{C}$ is hyperelliptic, and its unique hyperelliptic involution $\tilde{\sigma}$ fixes $\tilde{P}$ pointwise and interchanges the two preimages of every node of $C$.
            \item If $\tilde{g} = 1$: There exists an involution $\tilde{\sigma}$ on $\tilde{C}$ that fixes $\tilde{P}$ pointwise, interchanges the two preimages of every node of $C$, and yields a rational quotient $\tilde{C}/\langle\tilde{\sigma}\rangle \simeq \PP^1_k$.
            \item If $\tilde{g} = 0$: There exists an involution $\tilde{\sigma}$ on $\tilde{C} \simeq \PP^1_k$ that fixes $\tilde{P}$ pointwise and interchanges the two preimages of every node of $C$. Furthermore, if $\Char k \neq 2$, $\tilde{\sigma}$ does not act as the identity.
        \end{itemize}
        
        \item $C$ is a \emph{binary curve}, $n=0$, and the tuples of intersection nodes on the two components are projectively equivalent.
    \end{enumerate}
    In the irreducible case, $n$ is bounded by the number of fixed points of $\tilde{\sigma}$. Specifically, $n \leq 2\tilde{g} + 2$ if $\Char k \neq 2$. If $\Char k = 2$, then $n \leq \tilde{g} + 1$ unless $C$ is smooth of genus $0$, in which case $n$ is unbounded.
\end{theorem}

\begin{proof}
    The structural characterization follows directly from the proof of Theorem~\ref{thm:uniqueness_hyperelliptic}.

    In the irreducible case, the hyperelliptic involution fixes $P$ pointwise, so $n$ cannot exceed the number of fixed points of the lift $\tilde{\sigma}$ on $\tilde{C}$. The quotient map $\tilde{C} \to \tilde{C}/\langle\tilde{\sigma}\rangle \simeq \PP^1_k$ has degree $2$ unless $\sigma$ is the identity. By definition, this trivial action can occur if and only if $\Char k = 2$. In this scenario, the lift $\tilde{\sigma}$ cannot interchange the preimages of any nodes, which forces $C$ to be smooth, yielding $C \simeq \tilde{C} \simeq \PP^1_k$. For $\Char k \neq 2$, the Riemann--Hurwitz formula yields exactly $2\tilde{g} + 2$ ramification points, providing the bound $n \leq 2\tilde{g} + 2$. For $\Char k = 2$, assuming $C$ is not smooth of genus $0$, wild ramification restricts the number of fixed points of $\tilde{\sigma}$ to at most $\tilde{g} + 1$ via Artin--Schreier theory, yielding the bound $n \leq \tilde{g} + 1$.
\end{proof}

\subsection{Hyperellipticity in Characteristic 2}

We now clarify the divergent behavior of hyperellipticity in characteristic $2$. 

\begin{proposition} \label{prop:hyperelliptic_genus2}
    Let $(C, P)$ be a stable pointed curve equipped with an involution $\sigma$ that fixes $P$ pointwise and yields a rational tree quotient $C/\langle\sigma\rangle$.
    For an irreducible component $Z \subset C$, the restriction $\sigma|_Z = \id_Z$ if and only if $Z \simeq \PP^1_k$ and $Z$ intersects the rest of the curve exclusively at separating nodes.
    Consequently, $(C,P)$ is hyperelliptic if and only if either $\Char k = 2$ or $C$ does not contain any such component $Z$.
\end{proposition}

\begin{remark}
    Such a component is referred to as a \emph{loosely connected rational tail} in the terminology of \cite[Definition~3.2, Remark~3.3]{catanese1982pluricanonical}.
\end{remark}

\begin{proof}
    We may apply Proposition~\ref{prop:node_classification} and Theorem~\ref{thm:uniqueness_hyperelliptic} to $\sigma$, as their proofs do not rely on Definition~\ref{def:hyperelliptic}\ref{def:hyperelliptic_ii}.

    If $\sigma|_Z = \id_Z$, the quotient map embeds $Z$ into the rational tree $T$, forcing $Z \simeq \PP^1_k$. 
    Any node connecting $Z$ to $C \setminus Z$ is fixed with preserved branches, making it a separating node by Proposition~\ref{prop:node_classification}. 
    
    Conversely, if $Z \simeq \PP^1_k$ connects exclusively via separating nodes, $\sigma$ fixes these nodes and preserves $Z$. 
    The restriction $\sigma|_Z$ must fix the connecting nodes and $P \cap Z$. 
    Stability requires at least three such special points; since the only automorphism of $\PP^1_k$ fixing three points is the identity, $\sigma|_Z = \id_Z$. 
    The final claim follows from uniqueness and the characteristic-dependent condition in Definition~\ref{def:hyperelliptic}.
\end{proof}

\begin{example}\label{ex:char2_exclusive}
    Consider a stable curve $C$ of genus $3$ obtained by attaching three smooth elliptic curves $(E_i, p_i)_{i=1,2,3}$ to a smooth rational component $Z \simeq \PP^1_k$. 
    
    \begin{figure}[h!]
    \centering
    \begin{tikzpicture}[scale=1.2]
        % Curve C
        \begin{scope}[shift={(0,0)}]
            \draw (0, 0) -- (3.0, 0) node[right] {$Z$};
            \draw[line width = 1.5] (0.2, 0.8) node[above] {$E_1$} -- (0.8, -0.8);
            \draw[line width = 1.5] (1.5, 0.8) node[above] {$E_2$} -- (1.5, -0.8);
            \draw[line width = 1.5] (2.8, 0.8) node[above] {$E_3$} -- (2.2, -0.8);
            \filldraw (0.5, 0) circle (1.5pt) node[below left] {$p_1$};
            \filldraw (1.5, 0) circle (1.5pt) node[below right] {$p_2$};
            \filldraw (2.5, 0) circle (1.5pt) node[below right] {$p_3$};
        \end{scope}

        % Arrow
        \draw[->, thick] (3.8, 0) -- (5.2, 0) node[midway, above] {$/\langle \sigma \rangle$};

        % Quotient tree
        \begin{scope}[shift={(6.0,0)}]
            \draw (0, 0) -- (3.0, 0);
            \draw (0.2, 0.8) -- (0.8, -0.8);
            \draw (1.5, 0.8) -- (1.5, -0.8);
            \draw (2.8, 0.8) -- (2.2, -0.8);
            \filldraw (0.5, 0) circle (1.5pt);
            \filldraw (1.5, 0) circle (1.5pt);
            \filldraw (2.5, 0) circle (1.5pt);
        \end{scope}
    \end{tikzpicture}
    \caption*{Visualization of the quotient map $C \to C/\langle \sigma \rangle$. Thick lines indicate components of geometric genus $1$, while thin lines indicate geometric genus $0$.}
    \end{figure}
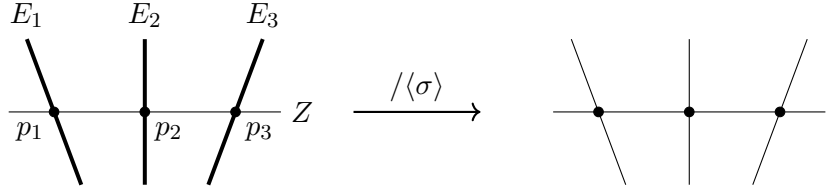

    Define an involution $\sigma \in \Aut_k(C)$ by taking the $[-1]$-involution on each elliptic curve $E_i$ (which fixes the node $p_i$) and the identity map on $Z$. Then the quotient $C/\langle \sigma \rangle$ is a rational tree and $Z$ is a smooth rational component that intersects the rest of the curve exclusively at the separating nodes $p_1, p_2$, and $p_3$. Proposition~\ref{prop:hyperelliptic_genus2} implies that $C$ is hyperelliptic if and only if $\Char k = 2$.
\end{example}

Conversely, wild ramification in characteristic $2$ can also prevent certain combinatorial types from being hyperelliptic.

\begin{example}\label{ex:char2_forbidden}
    Let $C$ be a stable curve of genus $3$ consisting of a central rational nodal curve $Z$ (i.e., a component of geometric genus $0$ with one self-node) attached to two smooth elliptic curves $(E_i, p_i)_{i=1,2}$.
    
    \begin{figure}[h!]
    \centering
    \begin{tikzpicture}[scale=0.85]
        % Curve C
        \begin{scope}[shift={(0,0)}]
            \draw (-1.0, 0.0) to[curve through={(1.0, 0.2) (1.2, 0.4) (1.0, 0.8) (0.8, 0.4) (1.0, 0.2)}] (3.0, 0.0);
            \node at (3.3, 0) {$Z$};
            \draw[line width = 1.5] (-0.5, 1.5) node[above] {$E_1$} -- (0.0, -0.5);
            \draw[line width = 1.5] (2.5, 1.5) node[above] {$E_2$} -- (2.0, -0.5);
            \filldraw (-0.085, -0.16) circle (2pt) node[below left] {$p_1$};
            \filldraw (2.085, -0.16) circle (2pt) node[below right] {$p_2$};
        \end{scope}

        % Arrow
        \node[scale=1.5] at (5.2, 0.3) {$\leadsto$};
        \node at (5.2, 0.8) {decompose};

        % 2-inseparable components
        \begin{scope}[shift={(9.5,0)}]
            \begin{scope}[shift={(0,0)}]
                \draw (-1.0, 0.0) to[curve through={(1.0, 0.2) (1.2, 0.4) (1.0, 0.8) (0.8, 0.4) (1.0, 0.2)}] (3.0, 0.0);
                \node at (3.3, 0) {$Z$};
                \filldraw (-0.085, -0.16) circle (2pt) node[below left] {$p_1$};
                \filldraw (2.085, -0.16) circle (2pt) node[below right] {$p_2$};
            \end{scope}
            \begin{scope}[shift={(-2.0, 0)}]
                \draw[line width = 1.5] (-0.5, 1.5) node[above] {$E_1$} -- (0.0, -0.5);
                \filldraw (-0.085, -0.16) circle (2pt) node[below left] {$p_1$};
            \end{scope}
            \begin{scope}[shift={(2.0, 0)}]
                \draw[line width = 1.5] (2.5, 1.5) node[above] {$E_2$} -- (2.0, -0.5);
                \filldraw (2.085, -0.16) circle (2pt) node[below right] {$p_2$};
            \end{scope}
        \end{scope}
    \end{tikzpicture}
    \caption*{Decomposition of $C$ into its three $2$-inseparable components. Thick lines indicate components of geometric genus $1$, while thin lines indicate geometric genus $0$.}
    \end{figure}
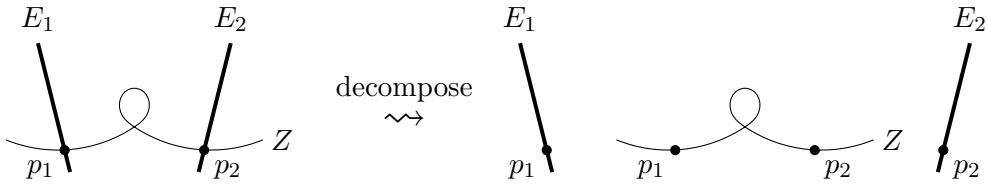

    We claim that if $\Char k = 2$, then $C$ is not hyperelliptic. Indeed, consider the $2$-inseparable component corresponding to the central curve $Z$. Its normalization has geometric genus $\tilde{g} = 0$ and carries two marked points corresponding to the intersections with the elliptic tails. By Theorem~\ref{thm:structure_hyperelliptic}, however, a hyperelliptic component in characteristic $2$ can possess at most $\tilde{g}+1 = 1$ marked points. Thus, $C$ is not hyperelliptic.
\end{example}

\subsection{A Characterization via the Canonical Sheaf}

A classical theorem for smooth projective curves of genus $g \geq 2$ states that $C$ is not hyperelliptic if and only if the canonical sheaf $\omega_{C/k}$ is very ample. This characterization extends naturally to $2$-inseparable stable curves.

\begin{theorem} \label{thm:canonical_embedding_general}
    Let $C$ be a $2$-inseparable stable curve. The canonical sheaf $\omega_{C/k}$ is very ample if and only if $C$ is not hyperelliptic.
\end{theorem}

This result essentially follows from \cite{catanese1982pluricanonical}. The details of this deduction will become clear in the proof of the more general version below, which is applicable to all $2$-inseparable stable pointed curves. To formulate this, we introduce the following definition:

\begin{definition}
    Let $(C, P)$ be a $2$-inseparable stable pointed curve. The associated \emph{pinched curve} $C'$ is constructed by replacing each marked point $p \in P$ with an ordinary cusp. Formally, we define $C'$ by replacing the local ring $\mathcal{O}_{C,p}$ with the subring $\mathcal{O}_{C',p} := k + \mathfrak{m}_p^2 \subset \mathcal{O}_{C,p}$ for every $p \in P$, where $\mathfrak{m}_p$ denotes the maximal ideal of $\mathcal{O}_{C,p}$.
\end{definition}

For the general theory of such local gluing and pinching constructions, see for example the classical treatment by Serre \cite[Chapter IV, \S 1]{serre}. By construction, $C'$ has only Gorenstein singularities, and hence its dualizing sheaf $\omega_{C'/k}$ is invertible. Furthermore, the natural map $\nu \colon C \to C'$---which acts as the identity on the underlying topological spaces and as the canonical inclusion on the local rings---is precisely the partial normalization at the cusps.

\begin{theorem}\label{thm:canonical_embedding_general_pointed}
Let $(C, P)$ be a $2$-inseparable stable pointed curve, let $C'$ be its associated pinched curve, and let $\nu \colon C \to C'$ denote the partial normalization at the cusps. Then the following conditions are equivalent:
\begin{enumerate}
    \item The stable pointed curve $(C, P)$ is hyperelliptic.
    \item There exists a finite morphism $f \colon C' \to \PP^1_k$ of degree $2$.
    \item The relative dualizing sheaf $\omega_{C'/k}$ is not very ample.
\end{enumerate}
\end{theorem}

\begin{proof}
Because $(C, P)$ is stable, the log-canonical sheaf $\omega_{C/k}(P)$ is ample, and thus so is the pullback $\nu^*\omega_{C'/k} \cong \omega_{C/k}(2P)$. Since $\nu$ is finite surjective, $\omega_{C'/k}$ itself is ample. In the language of \cite[Definition 0.1]{catanese1982pluricanonical}, this means $C'$ is ``canonically positive,'' which is the baseline assumption throughout \cite[Section 3]{catanese1982pluricanonical}. Furthermore, the $2$-inseparability of $C$ ensures $C'$ is ``very strongly connected'' \cite[Definition 3.21]{catanese1982pluricanonical}. We also note that its only singularities are ordinary nodes and cusps, which are double points.

We first establish (ii) $\iff$ (iii). In the terminology of \cite{catanese1982pluricanonical}, condition (ii) means exactly that $C'$ is ``honestly hyperelliptic'' \cite[Definition 3.18]{catanese1982pluricanonical}. We must relate this to Catanese's more general notion of a ``hyperelliptic'' Gorenstein curve \cite[Definition 3.9]{catanese1982pluricanonical}. By \cite[Theorem F]{catanese1982pluricanonical}, these two notions differ only if the curve contains a proper subcurve intersecting its complement in degree $2$. Since $C'$ has only double points, such an intersection would require two nodes to disconnect the curve, which contradicts the fact that $C'$ is very strongly connected. Thus, for $C'$, being honestly hyperelliptic is completely equivalent to being hyperelliptic.

Assume (ii) is false. By the equivalence just established, $C'$ is not hyperelliptic. Because $C'$ satisfies all the requirements of \cite[Theorem G]{catanese1982pluricanonical} (canonically positive, very strongly connected, not hyperelliptic, and only double points), we conclude that $\omega_{C'/k}$ is very ample. Conversely, if (ii) holds, $C'$ is hyperelliptic. By \cite[Proposition 3.10]{catanese1982pluricanonical}, the canonical map of such a curve is not even birational, and hence $\omega_{C'/k}$ is not very ample.

For (i) $\implies$ (ii), let $\sigma$ be the hyperelliptic involution on $(C, P)$. Because $C$ is $2$-inseparable, the quotient $C/\langle\sigma\rangle \simeq \PP^1_k$ is irreducible. Since this quotient consists of only a single component, $\sigma$ cannot act as the identity on a proper subcurve of $C$; therefore, if $\sigma$ acts as the identity locally, $C$ must be irreducible of genus $0$ and $\sigma = \mathrm{id}_C$ globally. Let $\pi \colon C \to \PP^1_k$ denote the quotient map, or the relative Frobenius if $\sigma = \mathrm{id}_C$. We show that $\pi$ factors through $\nu$, i.e. the local image of $\pi^\#$ at each $x \in P$ lies in $\mathcal{O}_{C', \nu(x)} = k + \mathfrak{m}_x^2$. Passing to formal completions $\widehat{\mathcal{O}}_{C, x} \cong k[[t]]$, this subring consists exactly of series with vanishing linear terms. If $\sigma = \mathrm{id}_C$, then $\Char k = 2$ and $\pi$ maps $t \mapsto t^2$, which trivially lies in this subring. If $\sigma \neq \mathrm{id}_C$, the image is the invariant ring $k[[t]]^\sigma$. For $\Char k \neq 2$, linearizing yields $\sigma(t) = -t$ and $k[[t]]^\sigma = k[[t^2]]$. For $\Char k = 2$, we have $\sigma(t) = t + h(t)$ with $0 \neq h(t) \in (t^2)$; for any invariant series $g(t) = \sum c_j t^j$, the relation $g(t + h(t)) - g(t) = 0$ yields a lowest-degree difference of $c_1 h(t) = 0$, forcing $c_1 = 0$. Thus, in all cases, $\pi$ factors locally, inducing the finite morphism $f \colon C' \to \PP^1_k$ of degree $2$.

For (ii) $\implies$ (i), the composition $\tilde{f} = f \circ \nu\colon C \to \PP^1_k$ is finite of degree $2$. If $C'$ is reducible, $f$ restricts to isomorphisms on exactly two components. These components must be smooth, precluding cusps ($P = \emptyset$); thus, $C'$ is a binary curve, and the component-swapping involution induced by $f$ defines the hyperelliptic involution. Assume $C'$ is irreducible. If $f$ is separable, it induces a Galois involution $\sigma'$ on $C'$. By the $2$-inseparability of $C'$, its canonical map is injective on the singular locus \cite[Theorem E]{catanese1982pluricanonical}. Since this map factors through $f$ \cite[Theorem F]{catanese1982pluricanonical}, $\sigma'$ must fix the singular locus, including the cusps. Lifting $\sigma'$ to the partial normalization yields an involution $\sigma$ on $C$ fixing the preimages $P$ pointwise, making $(C, P)$ hyperelliptic. If $f$ is inseparable, then $\Char k = 2$ and it is purely inseparable. In this case $\tilde{f}$ is a homeomorphism onto $\PP^1_k$. This implies $C$ can have no nodes; by semistability, it must therefore be entirely smooth, yielding $C \simeq \PP^1_k$. In this case, setting $\sigma = \id_C$ satisfies all conditions, making $(C, P)$ hyperelliptic.
\end{proof}

The assumptions of stability and $2$-inseparability in Theorem~\ref{thm:canonical_embedding_general_pointed} are automatically satisfied whenever the pinched curve is canonically embedded.

\begin{proposition} \label{prop:canonical_forces_stable}
    Let $C'$ be a connected, projective, reduced curve over $k$ whose only singularities are nodes and ordinary cusps. Assume its dualizing sheaf $\omega_{C'/k}$ is very ample. Let $S \subset C'$ be the set of its cusps, $\nu \colon C \to C'$ the partial normalization at $S$, and $P = \nu^{-1}(S)$. Then $C$ is a $2$-inseparable semistable curve, and the pointed curve $(C, P)$ is stable and non-hyperelliptic.
\end{proposition}

\begin{proof}
    Note that the connectivity terminology of Definition~\ref{def:comb_term} extends naturally to $C'$. Because $\omega_{C'/k}$ is very ample, the canonical linear system $|\omega_{C'/k}|$ is base-point free and induces a closed immersion. By \cite[Theorem D]{catanese1982pluricanonical}, any separating node is a base point of $|\omega_{C'/k}|$, so $C'$ cannot have any separating nodes. Furthermore, by \cite[Theorem E]{catanese1982pluricanonical}, the two nodes of a separating pair necessarily have the same image under the canonical map, which contradicts injectivity. Thus, $C'$ has no separating pairs either. Since cusps are unibranch, it follows that $C$ is $2$-inseparable.

    To show the pointed curve $(C,P)$ is stable, i.e., $\omega_{C/k}(P)$ is ample, we verify $\deg(\omega_{C/k}(P)|_Z) > 0$ for every irreducible component $Z \subset C$. Let $Y = \nu(Z)$ and $r_Y = |P \cap Z|$. The pullback of the dualizing sheaf satisfies $\nu^*(\omega_{C'/k}) \simeq \omega_{C/k}(2P)$. Restricting this to $Z$ and taking degrees yields $\deg(\omega_{C'/k}|_Y) = \deg(\omega_{C/k}(P)|_Z) + r_Y$.

    If $C'$ is irreducible, $Y=C'$. The canonical embedding $C' \hookrightarrow \PP^{g(C')-1}$ strictly requires $g(C') \geq 3$. Thus, $\deg(\omega_{C'/k}|_Y) = 2g(C') - 2 \geq 4$. Since $g(C') = g(Z) + r_Y$, we have $r_Y \leq g(C')$. The difference is $\deg(\omega_{C'/k}|_Y) - r_Y \geq g(C') - 2 \geq 1 > 0$.

    If $C'$ is reducible, $Y$ intersects the rest of the curve exclusively in $k_Y$ nodes. By the $2$-inseparability of $C'$, we must have $k_Y \geq 3$. Adjunction on $C'$ gives $\deg(\omega_{C'/k}|_Y) = 2g(Y) - 2 + k_Y$. Substituting $g(Y) = g(Z) + r_Y$ yields $\deg(\omega_{C'/k}|_Y) - r_Y = 2g(Z) + r_Y - 2 + k_Y$. Since $g(Z), r_Y \geq 0$ and $k_Y \geq 3$, this difference is strictly positive.

    Finally, $C'$ can be identified with the pinched curve associated to the $2$-inseparable stable curve $(C,P)$, and since $\omega_{C'/k}$ is very ample, we conclude with Theorem~\ref{thm:canonical_embedding_general_pointed} that $(C,P)$ is not hyperelliptic.
\end{proof}

\section{Families of Hyperelliptic Curves and Moduli Stacks} \label{sec:moduli}

In this section, we no longer assume our curves are defined over a fixed algebraically closed field. Instead, we extend our intrinsic definition of hyperellipticity to flat families over arbitrary base schemes. We formally define the moduli stack $\Hcbar_g$ of hyperelliptic stable curves as the closure of the smooth hyperelliptic locus $\Hc_g$ inside $\Mcbar_g$. Proceeding from this definition, we characterize the geometric points of $\Hcbar_g$, provide an explicit modular description over $\Spec \ZZ[1/2]$, and finally examine the combinatorial stratification of the boundary to highlight the divergent behavior in characteristic $2$.

\subsection{Families of Hyperelliptic Stable Curves}

\begin{definition} \label{def:family_stable_hyperelliptic}
    Let $S$ be a scheme and let $g \geq 2$ be an integer.
    \begin{enumerate}
        \item A \emph{stable curve of genus $g$ over $S$} \cite{deligne-mumford} is a flat, proper, finitely presented morphism $\pi\colon \mathcal{C} \to S$ whose geometric fibers $\mathcal{C}_{\overline{s}}$ are stable curves of genus $g$ over $\kappa(\overline{s})$ in the sense of Definition~\ref{def:stable_pointed_curve}.
        
        \item \label{def:family_stable_hyperelliptic_ii} A stable curve $\mathcal{C} \to S$ of genus $g$ is \emph{hyperelliptic} if it admits an involution $\sigma \in \Aut_S(\mathcal{C})$ such that on every geometric fiber $\mathcal{C}_{\overline{s}}$, the induced automorphism $\sigma_{\overline{s}}$ is a hyperelliptic involution in the sense of Definition~\ref{def:hyperelliptic}.
    \end{enumerate}
\end{definition}

The automorphism $\sigma$ is referred to as the \emph{(global) hyperelliptic involution}. We note that for smooth families $\mathcal{C} \to S$, this formulation coincides with \cite[Definition~5.4]{lonsted1979basics}. As an immediate consequence of our uniqueness result on the geometric fibers, we obtain the global uniqueness of this involution.

\begin{corollary} \label{cor:unique_hyp}
    The hyperelliptic involution of a hyperelliptic stable curve $\mathcal{C} \to S$ is unique.
\end{corollary}

\begin{proof}
    By Theorem~\ref{thm:uniqueness_hyperelliptic}, any two hyperelliptic involutions $\sigma$ and $\tau$ agree on every geometric fiber. The relative automorphism group scheme $\underline{\Aut}_S(\mathcal{C})$ is finite and unramified over $S$ \cite[Theorem~1.11]{deligne-mumford}. Consequently, its diagonal is both an open and a closed immersion. The locus in $S$ where $\sigma$ and $\tau$ agree is therefore an open and closed subscheme. Since it contains all topological points, it equals $S$ scheme-theoretically, hence $\sigma = \tau$.
\end{proof}

\subsection{Hyperelliptic Stable Curves as Limits}

The following result demonstrates that our definition is not overly restrictive: any stable limit of a smooth hyperelliptic curve is inherently a hyperelliptic stable curve. We note that this specialization property critically relies on Condition~\ref{def:hyperelliptic_ii} of Definition~\ref{def:hyperelliptic}.

\begin{proposition} \label{prop:specialization_hyperelliptic}
    Let $R$ be a discrete valuation ring with fraction field $K$ and residue field $k$. Let $\mathcal{C}$ be a stable curve over $R$ whose generic fiber $\mathcal{C}_K$ is smooth. Assume there exists an automorphism $\sigma \in \Aut_R(\mathcal{C})$ whose restriction to the generic fiber is a hyperelliptic involution. Then $\mathcal{C}$ is a hyperelliptic stable curve over $R$.
\end{proposition}

\begin{proof}
    For $\Char k \neq 2$, this is precisely \cite[Proposition~4.3]{kawaguchi2015rank}. For $\Char k = 2$, the result follows from \cite[Proposition~3.13]{maugeais2003relevement}.
\end{proof}

Conversely, every hyperelliptic stable curve admits a smoothing to a smooth hyperelliptic curve. This ensures that our definition is not overly broad and captures precisely the intended boundary objects.

\begin{proposition} \label{prop:existence_smooth_lift}
    Let $C$ be a hyperelliptic stable curve over an algebraically closed field $k$. There exists a complete discrete valuation ring $R$ with residue field $k$ and a hyperelliptic stable curve $\mathcal{C}$ over $R$ such that:
    \begin{enumerate}
        \item the special fiber $\mathcal{C}_k$ is isomorphic to $C$; and
        \item the generic fiber $\mathcal{C}_K$ is a smooth hyperelliptic curve.
    \end{enumerate}
\end{proposition}

\begin{proof}
    This is \cite[Théorème~5.4]{maugeais2003relevement} applied to the case $p=2$, noting that the Kummerian condition translates precisely to our Condition~\ref{def:hyperelliptic_ii}. Alternatively, for $\Char k \neq 2$, the existence of such an equivariant smoothing also follows from \cite[Proposition~2.2]{ekedahl1995boundary}, as Condition~\ref{def:hyperelliptic_ii} guarantees that the involution acts admissibly in the sense of \cite[Definition~1.2]{ekedahl1995boundary}.
\end{proof}

\subsection{The Moduli Stack of Hyperelliptic Stable Curves}

Recall that the moduli stack of stable curves of genus $g \geq 2$, denoted $\Mcbar_g$, is a smooth and proper Deligne--Mumford stack over $\Spec \mathbb{Z}$ \cite{deligne-mumford}. Its open substack $\Mc_g$ parametrizes families of smooth curves.

Because the hyperelliptic involution is unique (Corollary~\ref{cor:unique_hyp}), being hyperelliptic is an intrinsic property rather than additional data, and it naturally satisfies descent. Therefore, the full subcategory $\Hc_g$ of $\Mc_g$ spanned by smooth hyperelliptic curves constitutes a substack, the moduli stack of smooth hyperelliptic curves. It is well known that $\Hc_g$ is an open Deligne--Mumford substack of $\Mc_g$ and that it is smooth over $\Spec \ZZ[1/2]$ \cite{laudal2006deformations, ekedahl1995boundary}.

We define the moduli stack of hyperelliptic stable curves, denoted $\Hcbar_g$, as the scheme-theoretic closure of $\Hc_g$ inside $\Mcbar_g$. By definition, it is a proper Deligne--Mumford stack over $\Spec \ZZ$.

Although this abstract definition establishes the stack structure of $\Hcbar_g$, it does not explicitly characterize its geometric points. The following theorem provides this description.

\begin{theorem} \label{thm:geometric_points}
    Let $k$ be an algebraically closed field. A stable curve $C$ over $k$ represents a geometric point of the stack $\Hcbar_g$ if and only if $C$ is a hyperelliptic stable curve in the sense of Definition~\ref{def:hyperelliptic}.
\end{theorem}

\begin{proof}
    Because $\Mcbar_g$ is a locally Noetherian stack \cite[Theorem~5.1]{deligne-mumford}, the scheme-theoretic closure of $\Hc_g$ is determined by specialization over discrete valuation rings. The equivalence therefore follows immediately from the specialization and smoothing results of Proposition~\ref{prop:specialization_hyperelliptic} and Proposition~\ref{prop:existence_smooth_lift}.
\end{proof}

By inverting $2$, we obtain a fully explicit modular description of $\Hcbar_g$.

\begin{theorem}
    Over $\Spec \ZZ[1/2]$, the moduli stack $\Hcbar_g$ is smooth and coincides with the full subcategory of $\Mcbar_g$ consisting of hyperelliptic stable curves in the sense of Definition~\ref{def:family_stable_hyperelliptic}\ref{def:family_stable_hyperelliptic_ii}.
\end{theorem}

\begin{proof}
    Let $\Hcbar'_g$ denote the strictly full subcategory of $\Mcbar_g$ over $\Spec \ZZ[1/2]$ consisting of hyperelliptic stable curves. Because the property of being hyperelliptic is stable under base change, $\Hcbar'_g$ is a category fibered in groupoids. To establish that it forms a stack, it suffices to verify effective descent. If a stable curve becomes hyperelliptic over an \'etale cover, the global uniqueness of the involution (Corollary~\ref{cor:unique_hyp}) ensures that the local involutions automatically glue. Thus, $\Hcbar'_g$ is a stack.

    Let $\Mcbar_g(G)$ denote the stack classifying stable curves of genus $g$ equipped with an admissible action of $G = \ZZ/2\ZZ$ \cite[Definition~3.1]{ekedahl1995boundary}. By \cite[Theorem~3.2]{ekedahl1995boundary}, $\Mcbar_g(G)$ is a smooth and proper Deligne--Mumford stack over $\Spec \ZZ[1/2]$. Because the hyperelliptic involution is uniquely determined, equipping a curve with this canonical $G$-action identifies $\Hcbar'_g$ with a substack of $\Mcbar_g(G)$. Indeed, Definition~\ref{def:hyperelliptic}\ref{def:hyperelliptic_ii} guarantees that the hyperelliptic involution acts admissibly in the sense of \cite[Definition~1.2]{ekedahl1995boundary}. Furthermore, our definition requires the quotient curve to have arithmetic genus $0$. Since the arithmetic genus of the quotient is locally constant in flat families, this additional requirement imposes an open and closed condition. 
    
    Consequently, $\Hcbar'_g$ is isomorphic to a union of connected components of $\Mcbar_g(G)$, which implies that $\Hcbar'_g$ itself is a smooth and proper Deligne--Mumford stack. Finally, $\Hcbar'_g$ contains the locus of smooth hyperelliptic curves $\Hc_g$ as a dense open substack. As $\Hcbar'_g$ is a smooth, proper substack of $\Mcbar_g$ containing $\Hc_g$ as a dense open substack, it coincides with the scheme-theoretic closure $\Hcbar_g$.
\end{proof}

\begin{remark} \label{rem:char2_stack_failure}
    While Definition~\ref{def:family_stable_hyperelliptic} identifies the geometric points of $\Hcbar_g$ across all characteristics (Theorem~\ref{thm:geometric_points}), the category of such families fails to form an algebraic stack over $\Spec \ZZ$.

    To see this, let $\Hcbar'_3$ be the category fibered in groupoids of genus $3$ families satisfying Definition~\ref{def:family_stable_hyperelliptic}. Let $\mathcal{C} \to \Spec \ZZ_2$ be a stable curve constructed by attaching three smooth elliptic curves over $\Spec \ZZ_2$ to a central $\PP^1_{\ZZ_2}$. We equip $\mathcal{C}$ with an involution $\sigma \in \Aut_{\ZZ_2}(\mathcal{C})$ that acts as the identity on the central $\PP^1_{\ZZ_2}$ and as the standard hyperelliptic involution on the three elliptic tails.

    For $S_n = \Spec \ZZ/2^n\ZZ$, the unique geometric point of $S_n$ has characteristic $2$. Because Condition~\ref{def:hyperelliptic_ii} of Definition~\ref{def:hyperelliptic} is vacuous in characteristic $2$, the base change $\mathcal{C}_{S_n}$ belongs to $\Hcbar'_3(S_n)$ for all $n \ge 1$. If $\Hcbar'_3$ were algebraic, the effectiveness of formal objects (\cite[\href{https://stacks.math.columbia.edu/tag/07X8}{Lemma 07X8}]{stacks-project}) would force $\mathcal{C} \in \Hcbar'_3(\Spec \ZZ_2)$. This implies $\mathcal{C}_{\QQ_2} \in \Hcbar'_3(\Spec \QQ_2)$, contradicting Condition~\ref{def:hyperelliptic_ii}, which strictly prohibits the identity action on irreducible components in characteristic $0$.

    Consequently, definitions relying solely on geometric fiber conditions admit spurious nilpotent deformations. The scheme-theoretic closure $\Hcbar_g$ is therefore essential to properly encode the geometry over $\Spec \ZZ$. Providing an explicit modular description of the families $\mathcal{C} \to S$ corresponding to arbitrary morphisms $S \to \Hcbar_g$ over $\Spec \ZZ$ appears to be highly non-trivial. From the alternative point of view where we consider a smooth hyperelliptic curve as a ramified cover of degree $2$ of the projective line, such a modular description has recently been provided by Hippold \cite{hippold2026logarithmic}.
\end{remark}

\subsection{Enumeration of Combinatorial Types} \label{subsec:counting}

It is natural to investigate the stratification of the boundary of $\Hcbar_g$ by combinatorial type. Restricting our attention again to $\Spec \ZZ[1/2]$, it is a straightforward computational exercise to filter the combinatorial types of $\Mcbar_g$ (using, for instance, the SageMath package \texttt{admcycles} \cite{delecroix2022admcycles}) to identify those of semicompact type whose $2$-inseparable components satisfy the conditions of Theorem~\ref{thm:structure_hyperelliptic}. The fact that all such types indeed occur is established in \cite[Theorem~1.12]{kawaguchi2015rank}. For $g \geq 2$, this enumeration yields the sequence:
\[
    7, 32, 190, 1350, 10765, \dots
\]
(The first term is expected, as $\Mc_2 = \Hc_2$.) This sequence appears as a subsequence of OEIS \href{https://oeis.org/A007827}{A007827}. Indeed, this count can be explained theoretically via the well-known morphism of stacks
\[
    \Hcbar_g \longrightarrow [\overline{\mathcal{M}}_{0, 2g+2} / S_{2g+2}],
\]
which maps a hyperelliptic stable curve to the target of its associated admissible cover---a stable rational tree equipped with exactly $2g+2$ unordered branch points. Because this morphism is a $\mu_2$-gerbe \cite{abramovich2003twisted}, it induces a bijection on geometric points. The aforementioned OEIS sequence, for $n \geq 3$, can be seen as counting the possible combinatorial types of $[\overline{\mathcal{M}}_{0, n} / S_{n}]$, which naturally yields the subsequence above when evaluated at $n = 2g+2$. We note that this can equivalently be interpreted as the number of cluster pictures (see \cite[Section 9]{dokchitser2019semistable}).

In characteristic $2$, however, this combinatorial behavior diverges significantly. For instance, in genus $g=3$, there are four combinatorial types that are hyperelliptic \emph{exclusively} in characteristic $2$. These correspond to the curve studied in Example~\ref{ex:char2_exclusive} and its variations where the elliptic tails are replaced by rational nodal curves. Conversely, there are three combinatorial types that are strictly \emph{forbidden} in characteristic $2$ despite being hyperelliptic elsewhere; these correspond exactly to the curve from Example~\ref{ex:char2_forbidden} and its variations. A complete classification of these divergent hyperelliptic boundary strata for genus $3$ curves is provided in \cite[Theorem~1.10]{quartic_paper}, see also \cite[Proposition 2.45]{MaxMaster}.

More broadly, in characteristic $2$, wild ramification alters the branch locus and the aforementioned morphism breaks down. It is no longer a $\mu_2$-gerbe, and the natural bijection between admissible covers and hyperelliptic stable curves is consequently lost. For example, outside characteristic $2$, one can uniquely compute the stable reduction of a smooth genus $2$ curve by marking its Weierstrass points, computing the reduction of the marked curve, and then forgetting the markings and stabilizing. In characteristic $2$, this correspondence fails completely, and the reduction process becomes significantly more intricate, requiring the consideration of $54$ distinct intermediate cases \cite{gehrunger2025reduction}.

% --- Bibliography ---
\renewcommand{\bibfont}{\small}
\printbibliography

\end{document}